\documentclass[11pt,a4paper,reqno]{amsart}
\usepackage{amsthm,amsmath,amsfonts,amssymb,amsxtra,appendix,bookmark,dsfont,latexsym,bm,hyperref,delarray,color,euscript,amsgen,amsbsy,amsopn,amscd,latexsym,mathrsfs}

\makeatletter

\usepackage{hyperref}

\makeatletter
\def\@tocline#1#2#3#4#5#6#7{\relax
  \ifnum #1>\c@tocdepth 
  \else
    \par \addpenalty\@secpenalty\addvspace{#2}%
    \begingroup \hyphenpenalty\@M
    \@ifempty{#4}{%
      \@tempdima\csname r@tocindent\number#1\endcsname\relax
    }{%
      \@tempdima#4\relax
    }%
    \parindent\z@ \leftskip#3\relax \advance\leftskip\@tempdima\relax
    \rightskip\@pnumwidth plus4em \parfillskip-\@pnumwidth
    #5\leavevmode\hskip-\@tempdima
      \ifcase #1
       \or\or \hskip 1em \or \hskip 2em \else \hskip 3em \fi%
      #6\nobreak\relax
      \dotfill
      \hbox to\@pnumwidth{\@tocpagenum{#7}}
    \par
    \nobreak
    \endgroup
  \fi}
\makeatother

\newtheorem{theorem}{Theorem}[section]
\newtheorem{lemma}[theorem]{Lemma}

\newtheorem{definition}[theorem]{Definition}
\newtheorem{properties}[theorem]{Properties}

\theoremstyle{remark}

\newcommand\R{{\ensuremath {\mathbb R} }}
\newcommand\C{{\ensuremath {\mathbb C} }}
\newcommand\N{{\ensuremath {\mathbb N} }}

\newcommand\1{{\ensuremath {\mathds 1} }}

\renewcommand\phi{\varphi}

\newcommand{\wto}{\rightharpoonup}
\newcommand{\cS}{\mathcal{S}}
\newcommand{\cP}{\mathcal{P}}

\newcommand{\cI}{\mathcal{I}}

\renewcommand{\epsilon}{\varepsilon}

\DeclareMathOperator{\tr}{{\rm Tr}}

\renewcommand{\geq}{\geqslant}
\renewcommand{\leq}{\leqslant}

\numberwithin{equation}{section}

\begin{document}

\title{On two properties of the Fisher information}

\author[N. Rougerie]{Nicolas Rougerie}
\address{Universit\'e Grenoble-Alpes \& CNRS,  LPMMC (UMR 5493), B.P. 166, F-38042 Grenoble, France}
\email{nicolas.rougerie@grenoble.cnrs.fr}

\date{August, 2020}

\begin{abstract}
Alternative proofs for the superadditivity and the affinity (in the large system limit) of the usual and some fractional Fisher informations of a probability density of many variables are provided. They are consequences of the fact that such informations can be interpreted as quantum kinetic energies. 
\end{abstract}

\maketitle

\tableofcontents

\section{Introduction}

The Fisher information of a symmetric probability measure of $N$ variables (see Definition \ref{def:Fisher} below) is known to be: 

\medskip

\noindent\textbf{A}. Superadditive: the information of the full measure is not smaller than the sum of the informations of marginals with $n$ and $N-n$ variables. 
 
\medskip

\noindent\textbf{B}. Affine linear in the large $N$ limit: for probability measures that have a limit $N\to \infty$, the natural limiting information (mean, or level-3, information) is an affine functional.
 
\medskip

\textbf{A} has been proved first in~\cite[Theorem~3]{Carlen-91b} and then, by another method, in~\cite[Lemma~3.7]{HauMis-14}. \textbf{B} is deduced in~\cite[Proposition~3]{Kiessling-12} from the corresponding property for the mean entropy, originating in~\cite{RobRue-67}. In~\cite[Section~5.3]{HauMis-14} a different proof is provided, based on a general abstract linearity lemma. The analogues for some fractional variants of Fisher's information are obtained in~\cite{Salem-19,Salem-19b}. 

This note provides alternative proofs of these properties. If the Fisher information of a probability density $\mu_N \in \cP_{\rm sym} (\R^{dN})$ is interpreted as a quantum kinetic energy for the quantum state $\vert \Psi_N \rangle \langle \Psi_N \vert$, orthogonal projector on the ``quantum wave-function'' $\Psi_N = \sqrt{\mu_N} \in L^2 _{\rm sym} (\R^{dN})$, both properties become quite natural.  Roughly speaking \textbf{A} is a consequence of the convexity of the kinetic energy as a function of $\vert \Psi_N\vert^2$ and \textbf{B} follows from its affinity as a function of $\vert \Psi_N \rangle \langle \Psi_N \vert$.

As for the motivations behind proving \textbf{A} and \textbf{B}, they mostly come from the study of mean-field limits of large systems of statistical mechanics, classical and quantum. Indeed, they do not seem to be of use in the theory of sufficient statistics, where the Fisher information originates. We however note that any quantity interpreted as an ``information'' should certainly satisfy \textbf{A}. \textbf{B} has been known to hold for the entropy for a long time~\cite{RobRue-67}, and it seems natural to ask its equivalent for the Fisher information.

Both properties are crucial to Kiessling's approach~\cite{Kiessling-12} of the mean-field limit of bosonic ground states (see~\cite[Appendix~A]{Rougerie-LMU,Rougerie-spartacus} for review), which involves interpreting the quantum kinetic energy as a classical Fisher information. In fact we exploit here the reverse of Kiessling's point of view.

In~\cite{FouHauMis-14}, Fisher information bounds are used to control the mean-field limit of a classical statistical mechanics system with stochastic diffusions (see~\cite{Hauray-hdr} for review). The method is general and has been adapted to other models, e.g. in~\cite{Salem-19,Salem-19b}. Briefly, the entropy production\footnote{I follow the physicists' convention that the entropy is produced, not dissipated.} along the flow is controlled first. The Fisher information is naturally linked to variations of the entropy\footnote{The Fisher information is the derivative of the entropy along the heat flow.}, and one then deduces a control of the former quantity. This has several important applications, one of which relies on \textbf{A} and \textbf{B} above. Briefly, \textbf{A} allows to pass to the large $N$ limit and control the mean information of limiting objects. Then \textbf{B} implies that the de Finetti-Hewitt-Savage mixing measure of the limit is concentrated on probability measures with finite Fisher information. Uniqueness theorems for the mean-field equation in the latter class can then be put to good use.  

\medskip

\noindent \textbf{Acknowledgments.} Thanks to Samir Salem, conversations with whom motivated the write-up of this note, and to Mathieu Lewin for discussions on related topics a while ago. Financial support was provided by the European Research Council (ERC) under the European Union's Horizon 2020 Research and Innovation Programme (Grant agreement CORFRONMAT No 758620).

\section{Definitions and Results}

We are concerned with classical mechanics states, symmetric probability measures $\mu_N \in \cP_{\rm sym} (\R^{dN})$ for the distribution of $N$ indistinguishable particles living say in $\R^d$. We will freely identify measures and their densities with respect to Lebesgue measure on $\R^{dN}$. Symmetric here means 
\begin{equation}\label{eq:symmetric}
\mu_N (x_1,\ldots,x_N) = \mu_N (x_{\sigma(1)}, \ldots, x_{\sigma(N)}) 
\end{equation}
for almost all $X_N = (x_1,\ldots,x_N) \in \R^{dN}$ and any permutation $\sigma$ of the $N$ indices. In the applications we have in mind, $X_N$ can be a collection of spatial coordinates (this is the case for bosonic mean-field limits as considered in~\cite{Kiessling-12}) \emph{or} a collection of velocity variables (this is the case for the applications in kinetic theory~\cite{HauMis-14,Salem-19,Salem-19b}). 

\newpage

The quantities of our interest are as given in the 

\begin{definition}[\textbf{Fisher informations}]\label{def:Fisher}\mbox{}\\
For $\mu_N \in \cP_{\rm sym} (\R^{dN})$ we define 

\medskip 

\noindent$\bullet$ the Fisher information 
 \begin{equation}\label{eq:Fisher minus}
\cI_1 [\mu_N] = \left\langle \sqrt{\mu_N} \Big| \sum_{j=1} ^N - \Delta_{x_j} \Big| \sqrt{\mu_N} \right\rangle_{L^2 (\R^{dN})}.
 \end{equation}

 \medskip 

\noindent$\bullet$ the fractional Fisher information of order $0<s<1$ 
\begin{equation}\label{eq:Fisher frac minus}
\cI_s [\mu_N] = \left\langle \sqrt{\mu_N}  \Big| \sum_{j=1} ^N \left(- \Delta_{x_j}\right)^{s}  \Big| \sqrt{\mu_N} \right\rangle_{L^2 (\R^{dN})} .
\end{equation}  
\end{definition}

\begin{proof}[Remarks]\mbox{}\\
\noindent\textbf{1.} The following equivalent definitions are well-known. For the usual Fisher information (the first one is maybe the most commonly used) we have
\begin{align}\label{eq:Fisher}
\cI_1 [\mu_N] &= \frac{1}{4} \int_{\R^{dN}} |\nabla \log \mu_N| ^2 \mu_N= \frac{1}{4}  \int_{\R^{dN}} \frac{\left|\nabla \mu_N\right| ^2}{\mu_N}=\int_{\R^{dN}} |\nabla \sqrt{\mu_N}| ^2 \nonumber \\
 &= \sum_{j=1} ^N \int_{\R^{dN}} |k_j|^{2} \left|\widehat{\sqrt{\mu_N}} (k) \right|^2 dk
 \end{align}
 and for the fractional Fisher information:
\begin{align}\label{eq:Fisher frac}
\cI_s [\mu_N] &= \sum_{j=1} ^N \int_{\R^{dN}} |k_j|^{2s} \left|\widehat{\sqrt{\mu_N}} (k) \right|^2 dk \nonumber \\
&= C_{d,s} N \int_{\R^{d(N+1)}} \frac{\left| \sqrt{\mu_N (x,x_1,\ldots,x_{N-1}) } - \sqrt{\mu_N (y,x_1\ldots,x_{N-1}) } \right| ^2}{|x-y| ^{d+s}} dxdy dx_1 \ldots dx_{N-1}.
\end{align}
Here hat-bearing functions stand for Fourier transforms and $C_{d,s}$ is a constant only depending on $d$ and $s$. That these various definitions are equivalent either follows from straightforward calculations or is proved in standard textbooks, such as~\cite{LieLos-01}. The precise value of the constant $C_{d,s}$ is of no concern to this note but can be found in~\cite{LieLos-01,Salem-19,Salem-19b}.

\medskip 
 
\noindent\textbf{2.} It follows from results of~\cite{BouBreMir-01,MasNag-78,BouBreMir-02,MazSha-02} that 
\begin{equation}\label{eq:Bourgain}
 \cI_1 [\mu_N] = C_d \lim_{s\uparrow 1} \,(1-s) C_{d,s} ^{-1} \cI_s [\mu_N] 
\end{equation}
with
$$ C_d = \left( \int_{S^{d-1}} \cos \theta \, d\sigma \right) ^{-1}.$$
Here $S^{d-1}$ is the euclidean sphere equipped with its Lebesgue measure $d\sigma$ and $\theta=\theta (\sigma)$ represents the angle of $\sigma$ with respect to the vertical axis. This implies that $\cI_1 \equiv \cI$ is a natural limit case of $\cI_s$ for $s\to 1$. 

\medskip 

\noindent\textbf{3.} Other types of fractional Fisher informations are discussed in~\cite{Toscani-15,Toscani-16,Toscani-18}, in connection with statistics, and~\cite{Salem-19}, in connection with the fractional heat flow. 

\medskip 

\noindent\textbf{4.} In~\cite{Salem-19b} the fractional Fisher information is defined with an extra ``cut-off'' 
$$ \cI_{s,\gamma} [\mu_N] = \sum_{j=1} ^N \left\langle \sqrt{\mu_N} | \chi(x_j) (-\Delta_{x_j})^s \chi(x_j) | \sqrt{\mu_N}  \right\rangle_{L^2} $$
with $\chi (x) = (1 + |x|^2) ^{2\gamma} $, $\gamma <0$. This does not significantly change the structure of the object, nor the proofs of the results below. Indeed the "kinetic energy" 
$$ \langle u \vert \chi  (-\Delta)^s \chi \vert u \rangle$$
still enjoys the properties we need (see Properties \ref{pro:kinetic}), of which convexity as a function of $\vert u \vert^2$ and affinity as a function of the orthogonal projector $\vert u \rangle \langle u \vert$ are the most crucial. We leave the adaptations to the reader.

%
%
%
\end{proof}

Define now, for any integer $n\leq N$, the $n$-th marginal/reduced density of a measure $\mu_N \in \cP_{\rm sym} (\R^{dN})$ as the probability measure on $\R^{dn}$ with density
\begin{equation}\label{eq:marginals}
\mu_N^{(n)} (x_1,\ldots,x_n) := \int_{\R^{d(N-n)}} \mu_N (x_1,\ldots,x_n,y_{n+1},\ldots,y_N) dy_{n+1} \ldots dy_N. 
\end{equation}
The first result we provide an alternative proof for is the superadditivity of the functionals from Definition~\ref{def:Fisher}:

\begin{theorem}[\textbf{Superadditivity of Fisher informations}]\label{thm:superadd}\mbox{}\\
Let $n < N$  be two integers and $\mu_N \in \cP_{\rm sym} (\R^{dN})$. We have that 
\begin{equation}\label{eq:superadd}
\cI_s [\mu_N] \geq \cI_s \left[\mu_N ^{(n)}\right] + \cI_s \left[\mu_N ^{(N-n)}\right] 
\end{equation}
for any $0<s\leq 1$.
\end{theorem}

\begin{proof}[Remarks]\mbox{}\\
 \noindent \textbf{1.} For $s=1$, Carlen proved this via a Minkowski-like inequality. Hauray-Mischler use a dual formulation of the Fisher information (not mentioned in the remark after Definition~\ref{def:Fisher}) to recover the result. For $s<1$ the result is obtained in~\cite{Salem-19b} using the last formulation in~\eqref{eq:Fisher frac}. As per~\eqref{eq:Bourgain}, this also implies the result for $s=1$. 
 
 \medskip 
 
 \noindent \textbf{2.} The short proof we provide uses standard tools of quantum mechanics: reduced density matrices and convexity of $\mu_N \mapsto \cI_s [\mu_N]$.
 
 \medskip

 \noindent\textbf{3.} A useful consequence is that, if $N$ is an integer times $n$  
\begin{equation}\label{eq:superadd bis}
 \frac{1}{N} \cI_s [\mu_N] \geq \frac{1}{n} \cI_s [\mu_N^{(n)}]. 
\end{equation}
\end{proof}

In statistical mechanics one is often interested in the limit of large particle numbers, $N\to \infty$ , maybe with other parameters of the model scaled appropriately. Then our classical states turn into symmetric probability measures over infinite sequences, $\mu \in \cP_{ \rm sym} (\R^{d\N})$. One can also be interested in the limits $N\to \infty$ with fixed $n$ of the marginals~\eqref{eq:marginals}. It can then be useful to have a notion of ``mean Fisher information'' (sometimes also refered to as level-3 information):

\begin{definition}[\textbf{Mean Fisher information}]\label{def:mean Fisher}\mbox{}\\
Let $\mu \in \cP_{\rm sym} (\R^{d\N})$ be a symmetric probability measure over sequences in $\R^d$. Equivalently\footnote{By a theorem of Kolmogorov.}, let $(\mu ^{(n)})_n$ be a sequence of symmetric probability measures over $\R^{dn}$ satisfying the consistency condition
$$ \left( \mu ^{(n+1)} \right) ^{(n)} = \mu ^{(n)}.$$
For $0<s\leq 1$ the mean (fractional) Fisher information of $\mu$ is 
\begin{equation}\label{eq:mean Fisher}
\cI_s [\mu] := \limsup_{n\to \infty} \frac{1}{n} \cI_s [\mu ^{(n)}]. 
\end{equation}
\end{definition}

The existence of the $\limsup$ follows from~\eqref{eq:superadd bis}. It is in fact possible to see that the $\limsup$ is both a $\sup$ and a $\lim$, using Theorem~\ref{thm:affine}.

The definition is in complete analogy with that of the mean entropy, originating in~\cite{RobRue-67}. Perhaps surprisingly, this functional is affine, just as the mean entropy. We shall give an alternative proof of the 

\begin{theorem}[\textbf{The mean Fisher information is affine}]\label{thm:affine}\mbox{}\\
Let $\mu \in \cP_{\rm sym} (\R^{d\N})$ and $P\in \cP (\cP (\R^d))$ be its unique de Finetti-Hewitt-Savage measure, i.e., for all $n\geq 0$
\begin{equation}\label{eq:deF measure}
 \mu^{(n)} = \int_{\cP (\R^d)} \rho ^{\otimes n} dP (\rho).
\end{equation}
Assume that there exists a locally bounded $V:\R^d \mapsto \R$ with $V(x) \underset{|x|\to \infty}{\longrightarrow} +\infty$ such that 
\begin{equation}\label{eq:trap}
 \int_{\R^d} V \mu ^{(1)} < \infty. 
\end{equation}
Then, the mean Fisher informations from Definition~\ref{def:mean Fisher} satisfy, for all $0<s\leq 1$ 
\begin{equation}\label{eq:affine}
\cI_s [\mu] = \int_{\cP (\R^d)} \cI_s [\rho] dP (\rho)
\end{equation}
with $\cI_s [\rho]$ as in Definition~\ref{def:Fisher} with $N=1$. 
\end{theorem}

\begin{proof}[Remarks]\mbox{}\\
\noindent \textbf{1.} In~\eqref{eq:trap} I demand a bit more ``confinement'' than in previous versions of the statement~\cite{Kiessling-12,HauMis-14,Salem-19b}. This is harmless in applications, for the theorem is meant to be applied to limits $N\to\infty$ of $N$-body classical states $\mu_N$. In order for a state with infinitely many particles to exist in the limit, a tightness argument of the type
$$ \int_{\R^d} V \mu_N ^{(1)} < \infty, \mbox{ independently of  } N$$
is usually needed.

\medskip

\noindent \textbf{2.} Kiessling~\cite{Kiessling-12} proved the $s=1$ case, using the better known~\cite{RobRue-67} affinity of the mean entropy 
$$ \cS [\mu] := - \limsup_{n\to \infty} \frac{1}{n} \int_{\R^{dn}} \mu ^{(n)} \log \mu^{(n)}.$$
The Fisher information is the derivative of the entropy along the heat flow. Since the latter is linear (as is differentiation), the affinity of the Fisher information follows. 

\medskip

\noindent \textbf{3.} Hauray and Mischler~\cite{HauMis-14} gave another proof for $s=1$. A first sanity check is to convince oneself that 
\begin{equation}\label{eq:ortho}
 \cI \left[ \frac{1}{2} \rho_1 ^{\otimes \infty} + \frac{1}{2} \rho_2 ^{\otimes \infty} \right] = \frac{1}{2}  \cI \left[ \rho_1 ^{\otimes \infty} \right] + \frac{1}{2} \cI\left[ \rho_2 ^{\otimes \infty} \right] 
\end{equation}
where $\rho^{\infty}$ is the measure over $\R^{d\N}$ with $n$-th marginal $\rho ^{\otimes n}$ for any $n$. The reason for this is that, if $\rho_1 \neq \rho_2$, $\rho_1 ^{\otimes N}$ becomes more and more alien (``orthogonal'') to $\rho_2 ^{\otimes N}$ for large $N$.

Very briefly, the proof of~\cite{HauMis-14} checks a more elaborate version of the ``partial affinity''~\eqref{eq:ortho}, and then applies a general abstract lemma implying full affinity.  

\medskip

\noindent \textbf{4.} The same strategy is applied to the $s<1$ case in~\cite{Salem-19b}. The abstract lemma applies mutatis mutandis, but the argument giving the partial affinity is different. The sanity check~\eqref{eq:ortho} can also be found in~\cite[Appendix~A]{Rougerie-LMU,Rougerie-spartacus}.  

\medskip

\noindent \textbf{5.} Salem proves in~\cite{Salem-19} that the variant mean fractional information\footnote{This is the derivative of the entropy along the fractional heat flow.} based on 
$$ \widetilde{\cI}_s [\mu_N] = N \int_{\R^{d(N+1)}} \frac{\Phi\left(\sqrt{\mu_N (x,x_1,\ldots,x_{N-1})}, \sqrt{\mu_N (y,x_1\ldots,x_{N-1}) } \right)}{|x-y| ^{d+s}} dxdy dx_1 \ldots dx_{N-1}$$
with 
$$ \Phi (x,y) = (x-y) (\log x - \log y)$$
enjoys similar properties as that we defined, in particular affinity. See~\cite[Remark~3.4]{Salem-19} for more comments on the relation between $\cI_s$ and $\widetilde{\cI}_s$ and their respective uses. In particular, for the applications of~\cite{Salem-19}, results on $\cI_s$ could serve as alternatives to those on $\widetilde{\cI}_s$. 
\end{proof}

\section{Proofs}

\subsection{Preliminaries} 

Our point of view in this note is to think quantum mechanically, that is, in terms of $L^2$ functions and operators acting on them, rather than in terms of probability measures. Pick $\mu_N \in \cP_{\rm sym} (\R^{dN})$. For it to have a finite Fisher information it must actually be a function. Define then 
\begin{equation}\label{eq:quantum}
\Psi_N = \sqrt{\mu_N}, \quad \Gamma_N = |\Psi_N \rangle \langle \Psi_N|   
\end{equation}
The first object above is a bosonic wave-function, namely $\Psi_N \in L^2_{\rm sym} (\R^{dN},\C)$ satisfies 
$$ 
\Psi_N (x_1,\ldots,x_N) = \Psi_N (x_{\sigma(1)}, \ldots, x_{\sigma(N)}) 
$$
in analogy with~\eqref{eq:symmetric}. The second object in~\eqref{eq:quantum} is the $L^2$-orthogonal projector on the complex linear span of $\Psi_N$. It is a bosonic state, i.e. a positive trace-class operator with trace $1$, acting on $L^2_{\rm sym} (\R^{dN},\C)$.  Note that we do not use the usual quantization of classical mechanics: we simply use that a classical state of position \emph{or} velocity variables can be directly embedded in a quantum formalism. 

We shall write Fisher informations as quantum kinetic energies of $\Psi_N$ or $\Gamma_N$ 
\begin{equation}\label{eq:quantum Fisher}
\cI_s [\mu_N] = \left\langle \Psi_N | H_N | \Psi_N \right\rangle_{L^2} = \tr\left( H_N \Gamma_N \right)  
\end{equation}
with 
$$ H_N = \sum_{j=1} ^N h_{x_j}, \quad h=(-\Delta) ^s$$
and $h_{x_j}$ acting on the variable $x_j$. We do not emphasize the dependence on $s$, for our proofs shall be based solely on the following

\begin{properties}[\textbf{Quantum kinetic energies}]\label{pro:kinetic}\mbox{}\\
The kinetic energy 
$$ L^2 (\R^d,\C) \ni u \mapsto \langle u | h | u \rangle \in \R^+ \cup \{+\infty\}$$
with $h$ as above is
\begin{enumerate}
 \item Positivity preserving
$$ \langle u | h | u \rangle \geq \langle |u|\, | h | |u| \rangle$$
\item Convex as a function of $|u|^2$: 
$$ L^1 (\R^d,\R^+) \ni \rho \mapsto \langle \sqrt{\rho} | h | \sqrt{\rho} \rangle \in \R^+ \cup \{+\infty\}$$
is convex.
\item With locally compact resolvent. For a locally bounded $V:\R^d \mapsto \R$ with 
$$V(x) \underset{|x|\to \infty}{\longrightarrow} +\infty$$
the operator $h + V$ has compact resolvent, i.e. $(h+V + c) ^{-1}$ is compact as an operator on $L^2$, where $c$ is a constant sufficiently large for the inverse to make sense. 
\end{enumerate}
\end{properties}

\begin{proof}[Remarks]\mbox{}
The first property for $s=1$ is just the straightforward (at least for smooth functions $u= \vert u \vert e^{i\varphi}$) identity 
$$ \vert\nabla u \vert ^2 = \vert \nabla \vert u \vert \vert  ^2 + \vert u \vert^2  \vert \nabla  \varphi  \vert^2,$$ 
 a particular case of the diamagnetic inequality~\cite[Theorem~7.12]{LieLos-01}. For $s>1$ it follows immediately from the last definition in~\eqref{eq:Fisher frac}, the triangle inequality, and~\eqref{eq:Bourgain}. Note that ``positivity preserving'' usually means something stronger (but also true in the case at hand), namely that the heat flow associated with $h$ preserves positivity of functions. The property we require usually goes hand-in-hand with this "true" positivity-preserving property.

The second property can be found in~\cite[Theorems~7.8 and 7.13]{LieLos-01}. The third property follows from the Sobolev compact embedding in $L^2$.
\end{proof}

We also recall the notion of reduced density matrix, extending that of marginal~\eqref{eq:marginals}. We define the $n$-th reduced density matrix $\Gamma_N ^{(n)}$ by a partial trace 
\begin{equation}\label{eq:red dens mat}
 \Gamma_N ^{(n)} := \tr_{n+1\to N} \Gamma_N.
\end{equation}
This is the operator on $L^2_{\rm sym} (\R^{dn})$ defined by the relation 
$$ 
\tr\left(A_n \Gamma_N^{(n)} \right) = \tr\left(A_n \otimes \1^{\otimes (N-n)} \Gamma_N \right)
$$
for any bounded operator $A_n$ on $L^2_{\rm sym} (\R^{dn})$. Note that $\Gamma_N$, as a trace-class (in particular, Hilbert-Schmidt) operator on $L^2$ has an integral kernel 
$$ \Gamma_N (x_1,\ldots,x_N;y_1,\ldots,y_N) = \overline{\Psi_N (y_1,\ldots,y_N)}\Psi_N (x_1,\ldots,x_N).$$
The integral kernel of $\Gamma_N ^{(n)}$ is then given, formally, as 
\begin{multline*}
 \Gamma_N ^{(n)} (x_1,\ldots,x_n;y_1,\ldots,y_n) = \\ 
 \int_{\R ^{d(N-n)}} \Gamma_N (x_1,\ldots,x_n,z_{n+1},\ldots,z_N;y_1,\ldots,y_n,z_{n+1},\ldots,z_N) dz_{n+1} \ldots dz_N.
\end{multline*}
One use of these objects is that, if $N$ is a multiple of $n$,
\begin{equation}\label{eq:red DM}
\cI_s [\mu_N] =  \tr\left( H_N \Gamma_N \right) = \frac{N}{n} \tr\left( H_n \Gamma_N^{(n)} \right)
\end{equation}
taking partial traces and using that $H_N$ is a sum of terms acting on one variable at a time.

\subsection{Superadditivity}\label{sec:superadd}

We now prove Theorem~\ref{thm:superadd}. The following considerations are very much in the spirit of the Hoffmann-Ostenhof$^2$ inequality~\cite{Hof-77,Lewin-ICMP}.

From~\eqref{eq:quantum Fisher} and~\eqref{eq:red dens mat} we have, similarly to in \eqref{eq:red DM}, 
\begin{equation}\label{eq:split}
\cI_s [\mu_N] = \tr\left( H_n \Gamma_N ^{(n)}\right) + \tr\left( H_{N-n} \Gamma_N ^{(N-n)}\right)  
\end{equation}
where $\mu_N$ and $\Gamma_N$ are related by~\eqref{eq:quantum}. Now, $\Gamma_N^{(n)}$ being a positive trace-class operator with unit trace, the spectral theorem implies the existence of an orthonormal basis $(u_j)$ of $L^2 (\R^{dn})$ such that 
$$ \Gamma_N ^{(n)} = \sum_{j} \lambda_j |u_j\rangle \langle u_j|$$
where the positive numbers $\lambda_j$ add to $1$. Thus 
$$ \tr\left( H_n \Gamma_N ^{(n)}\right) = \sum_j \lambda_j \langle u_j | H_n | u_j \rangle$$
and, using Items 1 and 2 in Properties~\ref{pro:kinetic} we get 
$$
\tr\left( H_n \Gamma_N ^{(n)}\right) \geq \sum_j \lambda_j \langle |u_j| | H_n | |u_j| \rangle\geq \left\langle \sqrt{\rho_n} | H_n | \sqrt{\rho_n} \right\rangle
$$
 with 
$$ \rho_n := \sum_{j} \lambda_j |u_j| ^2.$$
The proof is concluded by observing that 
\begin{equation}\label{eq:n density}
\rho_n = \mu_N^{(n)} 
\end{equation}
with $\mu_N^{(n)}$ the $n$-th marginal of $\mu_N$ and arguing similarly for the second term of~\eqref{eq:split}. 

To see the truth of~\eqref{eq:n density}, identify any bounded function $V_n$ of $n$ variables in $\R^d$ with the corresponding multiplication operator on $L^2 (\R^{dn})$. Then  
\begin{align*}
\int_{\R^{dn}} V_n \rho_n &= \tr\left( V_n \Gamma_N ^{(n)}\right)\\
&= \tr\left( V_n \otimes \1^{\otimes (N-n)} \Gamma_N \right)\\
&= \left\langle \Psi_N | V_n \otimes \1^{\otimes (N-n)} | \Psi_N \right\rangle\\
&= \int_{\R^{dN}} V_n (x_1,\ldots,x_n) |\Psi_N (x_1,\ldots,x_N)| ^2 dx_1,\ldots dx_N\\
&= \int_{\R^{dn}} V_n \mu_N ^{(n)}.
\end{align*}
This completes the proof.

\subsection{Affinity}

Our proof of Theorem~\ref{thm:affine} is based on the manifest affinity of~\eqref{eq:quantum Fisher} as a function of $\Gamma_N$, and the quantum de Finetti theorem, a generalization of the classical de Finetti-Hewitt-Savage theorem, see~\cite{Rougerie-spartacus,Rougerie-LMU} for review. Readers familiar with the classical theorem could note that the version of the quantum theorem I use has a proof (see~\cite{HudMoo-75} and~\cite[Appendix~A]{LewNamRou-14}) essentially identical to that of Hewitt-Savage~\cite{HewSav-55}.

First observe that, as per Item 2 in Properties~\ref{pro:kinetic}, Definition~\ref{def:mean Fisher} and~\eqref{eq:deF measure} we immediately have 
that 
$$\cI_s [\mu] \leq \int_{\cP (\R^d)} \cI_s [\rho] dP (\rho).$$ 
We aim at a corresponding lower bound.

Define, for any $N\in\N$, a bosonic quantum wave-function and a bosonic quantum state as 
$$ \Psi_N = \sqrt{\mu ^{(N)}}, \quad \Gamma_N =|\Psi_N \rangle \langle \Psi_N|$$
where $\mu^{(N)}$ is the $N$-th marginal of $\mu\in\cP_{\rm sym} (\R^{dN})$. Let then, for $k\leq N$, $\Gamma_N^{(k)}$ be the reduced density matrix~\eqref{eq:red dens mat} of $\Gamma_N$. At fixed $k$, the sequence $(\Gamma_N^{(k)})_N$ is by definition bounded in the trace-class. Modulo a (not-relabeled) subsequence we thus have 
\begin{equation}\label{eq:weak star CV}
\Gamma_N ^{(k)} \wto_\star \gamma^{(k)}
\end{equation}
in the trace-class, as $N\to \infty$. The latter being ~\cite{Simon-79,Schatten-60} the dual of the compact operators (equiped with the operator norm, and where the duality bracket between two operators $A,B$ is given by $\tr ( A B )$), this means 
\begin{equation}\label{eq:weak star}
\tr\left( K_k \Gamma_N ^{(k)} \right) \to \tr\left( K_k \gamma ^{(k)} \right)
\end{equation}
for any compact operator $K_k$ on $L^2_{\rm sym} (\R^{dk})$. By a diagonal extraction argument we can assume that, for all $k\geq 0$ the sequences $(\Gamma_N^{(k)})_N$ converges weakly-$\star$, along a common subsequence in $N$. 

Let $V$ be the potential such that~\eqref{eq:trap} holds and 
$$ h^V = h + V.$$
Then
\begin{multline*}
\tr \left( \sum_{j=1} ^k h_{x_j} ^V \Gamma_N^{(k)}\right) = \frac{k}{N} \left\langle \sqrt{\mu^{(N)}} \big| \sum_{j=1} ^N h_{x_j} + V(x_j) \big| \sqrt{\mu ^{(N)}}   \right\rangle_{L^2} \\ \leq k \left(\cI_s [\mu] + \int_{\R^d} \mu^{(1)} V \right) < \infty
\end{multline*}
by assumption. Hence, for all $k\geq 0$ and any constant $c>0$
$$ \tr \left( \sum_{j=1} ^k (h+V+c)_{x_j} \Gamma_N^{(k)}\right) \leq C_{k,c}$$
independently of $N$. Using cyclity of the trace, the positive operator $ L_k ^{1/2} \Gamma_N^{(k)} L_k ^{1/2} $
with $L_k = \sum_{j=1} ^k (h+V+c)_{x_j}$ is bounded in the trace-class. As above this implies that (modulo a futher extraction of a subsequence) it converges weakly-$\star$ (in the sense of \eqref{eq:weak star})
\begin{equation}\label{eq:weak star CV bis}
L_k ^{1/2} \Gamma_N^{(k)} L_k ^{1/2} \wto_\star L_k ^{1/2} \gamma ^{(k)} L_k ^{1/2}
\end{equation}
where the limit is identified by testing the convergence with smooth finite-rank operators and recalling \eqref{eq:weak star CV}. As per Item 3 in Properties~\ref{pro:kinetic}, $h+V+c$ has compact inverse provided $c$ is chosen large enough. Consequently, so does $L_k$.  Then, by  \eqref{eq:weak star CV bis},
$$ \tr \Gamma_N^{(k)} = \tr \left( L_k ^{-1} L_k ^{1/2} \Gamma_N^{(k)} L_k ^{1/2} \right) \to \tr \left( L_k ^{-1} L_k ^{1/2} \gamma ^{(k)} L_k ^{1/2} \right) = \tr \left(  \gamma^{(k)}\right).$$
This proves convergence of the trace-class norm
$$ 1 = \tr \left( \Gamma_N^{(k)}\right) \to \tr \gamma^{(k)}$$
and hence that actually 
\begin{equation}\label{eq:strong CV}
 \Gamma_N ^{(k)} \to \gamma^{(k)} 
\end{equation}
strongly in the trace-class norm (see~\cite{dellAntonio-67} or ~\cite[Addendum~H]{Simon-79}). 

We now use the strong\footnote{"Strong" refers to the fact that we use strong trace-class convergence of reduced density matrices. This is the quantum analogue of the Hewitt-Savage theorem where one uses tightness of the marginals. There is also a "weak" quantum de Finetti theorem, relying only on weak-$\star$ convergence of reduced density matrices. We do not use it here.} quantum de Finetti theorem (see~\cite{HudMoo-75,LewNamRou-14} and~\cite{Rougerie-LMU,Rougerie-spartacus} for review)
to obtain the existence and uniqueness of a probability measure $Q$ over the unit sphere of $L^2 (\R^d)$ such that 
\begin{equation}\label{eq:quantum deF}
 \gamma^{(k)} = \int |u ^{\otimes k} \rangle \langle u^{\otimes k}|  dQ (u). 
\end{equation}
By~\eqref{eq:red DM} (with $n=1$)
$$
\frac{1}{N}\cI_s [\mu ^{(N)}] = \tr\left( h \Gamma_N ^{(1)} \right)
$$
and a lower semi-continuity argument gives 
$$ \liminf_{N\to \infty} \tr\left( h \Gamma_N ^{(1)} \right) \geq \tr \left( h \gamma^{(1)} \right)$$
so that, combining the two observations and~\eqref{eq:quantum deF},  
$$ \cI_s [\mu] \geq \int \langle u | h | u \rangle  dQ (u). $$
Recalling Item 1 of Properties~\ref{pro:kinetic} this yields 
$$ \cI_s [\mu] \geq  \int \langle |u| | h | |u|  \rangle  dQ (u) = \int \left\langle \sqrt{|u|^2} \,\big|\, h \, \big| \,\sqrt{|u|^2} \right\rangle  dQ (u)   $$
and the proof will be complete once we have proven the next lemma: 

\begin{lemma}[\textbf{Identification of de Finetti measures}]\mbox{}\\
Let $P\in\cP(\cP(\R^d))$ be the classical de Finetti measure defined in Theorem~\ref{thm:affine} and $Q\in \cP (L^2 (\R^d))$ the quantum de Finetti measure defined in~\eqref{eq:quantum deF}. We have, for any bounded continuous function $\Phi$ over $\cP(\R^d)$
$$ \int \Phi (\rho) dP (\rho) = \int \Phi(|u|^2) dQ (u).$$
\end{lemma}

\begin{proof}
The monomial functions $\Phi$ of the form (see~\cite[Section~1.7]{Golse-13})
$$ \Phi (\rho) = M_{k,\phi_k} (\rho) = \int_{\R^{dk}} \phi_k \rho ^{\otimes k}$$
with $k\in \N$ and $\phi_k$ bounded continuous over $\R ^{dk}$ generate a subalgebra of the continuous functions on the space of probability measures. This subalgebra is dense by the Stone-Weierstrass theorem and it thus suffices to test the claim against all the above monomials.

We identify the function $\phi_k$ with the multiplication operator thereby to write
\begin{align*}
 \int M_{k,\phi_k} (\rho) dP (\rho) &= \int_{\R^{dk}} \phi_k \mu^{(k)}\\
 &= \int_{\R^{dN}} \phi_k (x_1,\ldots,x_k) \mu^{(N)} (x_1,\ldots,x_N) dx_1\ldots dx_N\\
 &= \tr \left( \phi_k \otimes \1 ^{\otimes (N-k)} \Gamma_N \right)\\
 &= \tr \left( \phi_k \Gamma_N ^{(k)} \right)\to \tr \left( \phi_k \gamma ^{(k)} \right)
\end{align*}
using Equations~\eqref{eq:deF measure}-\eqref{eq:marginals}-\eqref{eq:quantum}-\eqref{eq:red dens mat} and, in the last step, the fact that multiplication by a bounded function is a bounded operator to pass to the limit using~\eqref{eq:strong CV}. Since the left-hand side actually does not depend on $N$ we deduce   
$$ \int M_{k,\phi_k} (\rho) dP (\rho) = \tr \left( \phi_k \gamma ^{(k)} \right). $$
As per~\eqref{eq:quantum deF} this implies the desired 
$$ \int M_{k,\phi_k} (\rho) dP (\rho) = \int M_{k,\phi_k} (|u|^2) dQ (u).$$
This being true for all $k$ and $\phi_k$, the proofs of the lemma and the theorem (approximating $\rho \mapsto \left\langle \sqrt{\rho} \,\big|\, h \, \big| \,\sqrt{\rho} \right\rangle$ by a sequence of continuous functions) are both complete.
\end{proof}

%

\end{document}